\documentclass[11pt]{article}
\usepackage{amsthm,amsmath,amssymb,bbm}

\newtheorem{thm}{Theorem}
\newtheorem{prop}[thm]{Proposition}
\newtheorem{lem}[thm]{Lemma}
\newtheorem{cor}[thm]{Corollary}

\def\Ex{{\mathbb E}}
\def\Pr{{\mathbb P}}
\def\ve{\varepsilon}

\def\ind{\mathbbm{1}}

\title{$L_1$-norm of combinations of products of independent random variables}
\author{Rafa{\l} Lata{\l}a\thanks{Research supported by the NCN grant DEC-2012/05/B/ST1/00412}}
\date{}

\begin{document}

\maketitle

\begin{abstract}
We show that $L_1$-norm of linear combinations (with scalar or vector coefficients) of products of i.i.d.\ nonnegative mean one
random variables is comparable to  $l_1$-norm of coefficients.
\end{abstract}

\section{Introduction and Main Results}

Let $X,X_1,X_2,\ldots$ be i.i.d. nonnegative r.v.'s such that $\Ex X=1$ and
$\Pr(X=1)< 1$. Define 
\begin{equation}
\label{defR}
R_0:=1 \quad \mbox{ and }\quad 
R_i:=\prod_{j=1}^{i}X_j \ \mbox{ for }i=1,2,\ldots.
\end{equation}

Obviously $\Ex R_i=1$ and therefore for any $a_0,a_1,\ldots,a_n$,
\begin{equation}
\label{trivial}
\Ex\bigg|\sum_{i=0}^n a_iR_i\bigg|\leq \sum_{i=0}^n |a_i|.
\end{equation}

Micha{\l} Wojciechowski (personal communication) asked whether inequality \eqref{trivial} may be reversed in the case
when  $X=1+\cos(Y)$, where $Y$ has the uniform distribution on $[0,2\pi]$. 
In \cite{MW} he showed that for such variables there exist sequences $(a_i)$ such that $|a_i|\leq 1$,
$|\sum_{i=0}^k a_i|\leq C$ for all $k\leq n$ and 
$\Ex|\sum_{i=0}^n a_iR_i|\geq cn$. Resently he posed a more general problem.

\medskip

\noindent
{\bf Problem.} Is it true that for any i.i.d. sequence as above estimate \eqref{trivial} may be reversed, i.e.
there exists a constant $c>0$ that depends only on the distribution of $X$ such that
\[
\Ex\bigg|\sum_{i=0}^n a_iR_i\bigg|\geq c\sum_{i=0}^n |a_i| \quad
\mbox{ for any } a_0,\ldots,a_n?
\]

The aim of this note is to give an affirmative answer to the Wojciechowski question even in the more general situation
of coefficients in a normed space $(F,\|\ \|)$.

First we study a simpler case when $X$ takes with positive probability values close to zero.
We  prove a more general result that does not require the identical distribution assumption.
Namely we consider  sequences $(X_i)$ satisfying the following assumptions:
\begin{equation}
\label{ass1}
X_1,X_2,\ldots \mbox{ are independent, nonnegative r.v's with mean one,} 
\end{equation}
\begin{equation}
\label{ass2}
\Ex \sqrt{X_i}\leq \lambda<1 \quad \mbox{ and }\quad  \Ex|X_i-1|\geq \mu>0 \quad \mbox{ for all } i. 
\end{equation}

Notice that if $X$ is a nondegenerate nonnegative random variable, then $\Ex\sqrt{X}<\sqrt{\Ex X}$
and $\Ex|X-1|>0$, hence \eqref{ass2} holds for i.i.d. mean one nonnegative sequences.

\begin{thm}
\label{thm:noniid1}
Let $R_i$ be as in \eqref{defR}, where $X_1,X_2,\ldots$ satisfy assumptions \eqref{ass1} and \eqref{ass2}.
Then for any coefficients $v_0,v_1,\ldots,v_n$ in a normed space $(F,\|\ \|)$ we have
\[
\Ex\bigg\|\sum_{i=0}^nv_iR_i\bigg\|\geq c\sum_{i=0}^n\|v_i\|,
\] 
where
\[
c=\frac{1}{64}\min\{\mu,1\}\min_{1\leq i\leq n}\Pr\Big(X_i\leq \frac{(1-\lambda)^2}{256}\min\{\mu,1\}\Big).
\]
\end{thm}

Theorem \ref{thm:noniid1} immediately yields the following.

\begin{cor}
\label{iidprod1}
Let $X,X_1,X_2,\ldots$ be an i.i.d. sequence of nonnegative r.v's such that $\Ex X=1$ and $\Pr(X\leq \ve)>0$ for
any $\ve >0$. Then there exists a constant $c$ that depends only on the distribution of $X$ such that
for any $v_0,v_1,\ldots,v_n$ in a normed space $(F,\|\ \|)$,
\[
\Ex\bigg\|\sum_{i=0}^nv_iR_i\bigg\|\geq c\sum_{i=0}^n \|v_i\|.
\]
\end{cor}

\medskip
\noindent
{\bf Example} In the case related to Riesz products,  when $X_1,X_2,\ldots$ are independent with the same distribution as
$1+\cos(Y)$ with $Y$ uniformly distributed on $[0,2\pi]$ we have
\[
\lambda=\Ex\sqrt{1+\cos(Y)}=\sqrt{2}\Ex\Big|\cos\Big(\frac{Y}{2}\Big)\Big|=\frac{2\sqrt{2}}{\pi},\quad
\mu=\Ex|\cos(Y)|=\frac{2}{\pi}
\]
and (since $\cos x\geq 1-x^2/2$) for $0<\ve<1/2$,
\[
\Pr(X_i\leq \ve)=\Pr(\cos(Y)\geq 1-\ve)\geq \frac{\sqrt{2\ve}}{\pi}. 
\]
Thus the constant given by Theorem \ref{thm:noniid1} in this case is 
$c\geq \frac{1}{256}\pi^{-5/2}(1-\frac{2\sqrt{2}}{\pi})\geq 2\cdot 10^{-5}.$

\medskip

To treat the general case we need one more assumption that basically states that the most of the mass of $X_i$'s lies 
in the interval $[0,A]$.

\begin{equation}
\label{ass3}
\Ex|X_i-1|\ind_{\{X_i\geq A\}}\leq \frac{1}{4}\mu \quad \mbox{ for all } i. 
\end{equation}

\begin{thm}
\label{thm:noniid2}
Let $X_1,X_2,\ldots$ satisfy assumptions \eqref{ass1}, \eqref{ass2} and \eqref{ass3}. 
Then for any vectors $v_0,v_1,\ldots,v_n$ in a normed space $(F,\|\ \|)$, we have
\[
\Ex\bigg\|\sum_{i=0}^nv_iR_i\bigg\|\geq \frac{1}{512k}\mu^3\sum_{i=0}^n\|v_i\|,
\] 
where $R_i$ are as in \eqref{defR} and $k$ is a positive integer such that
\begin{equation}
\label{def:k}
\frac{2^{17}}{(1-\lambda)^2}k\lambda^{2k-2}A\leq \mu^3.
\end{equation}
\end{thm}

Since in the i.i.d. case all assumptions are clearly satisfied we get the positive answer to Wojciechowski's question.

\begin{thm}
\label{thm:iidprod2}
Let $X,X_1,X_2,\ldots$ be an i.i.d. sequence of nonnegative nondegenerate r.v's such that $\Ex X=1$. 
Then there exists a constant $c$ that depends only on the distribution of $X$ such that
for any $v_0,v_1,\ldots,v_n$ in a normed space $(F,\|\ \|)$,
\[
\Ex\bigg\|\sum_{i=0}^nv_iR_i\bigg\|\geq c\sum_{i=0}^n \|v_i\|.
\]
\end{thm}

In the symmetric case the similar estimate follows by conditioning.

\begin{cor}
Let $X,X_1,X_2,\ldots$ be an i.i.d. sequence of symmetric r.v's such that $\Ex|X|=1$ and $\Pr(|X|=1)<1$. 
Then there exists a constant $c$ that depends only on the distribution of $X$ such that
for any $v_0,v_1,\ldots,v_n$ in a normed space $(F,\|\ \|)$,
\[
\Ex\bigg\|\sum_{i=0}^nv_iR_i\bigg\|\geq c\sum_{i=0}^n \|v_i\|.
\]
\end{cor}

\begin{proof}
Let $(\ve_i)$ be a sequence of independent symmetric $\pm 1$ r.v's independent of $(X_i)$. Then by
Theorem \ref{thm:iidprod2}
\begin{align*}
\Ex\bigg\|\sum_{i=0}^nv_iR_i\bigg\|
&=
\Ex_{\ve}\Ex_X\bigg\|v_0+\sum_{i=1}^nv_i\prod_{k=1}^i\ve_k\prod_{k=1}^i|X_k|\bigg\|
\\
&\geq \Ex_{\ve}c\Big(\|v_0\|+\sum_{i=1}^n\Big\|v_i\prod_{k=1}^i\ve_k\Big\|\Big)
=c\sum_{i=0}^n \|v_i\|.
\end{align*}
\end{proof}

\medskip

\noindent
{\bf Example.} Assumption $\Pr(|X|=1)<1$ is crucial since
\[
\Ex\Big|\sum_{i=1}^n\prod_{k=1}^i\ve_k\Big|=\Ex\Big|\sum_{i=1}^n\ve_i\Big|\leq
\Big(\Ex\Big|\sum_{i=1}^n\ve_i\Big|^2\Big)^{1/2}=n^{1/2}.
\]

\medskip

Let  $(n_k)_{k\geq 1}$ be an increasing sequence of positive integers such that $n_{k+1}/n_k\geq 3$. 
Riesz products are defined by
\begin{equation}
\bar{R}_i(t)=\prod_{j=1}^i(1+\cos(n_jt)),\quad i=1,2,\ldots.
\end{equation}
It is well known that if $n_k$ grow sufficiently fast then 
$\|\sum_{i=0}^na_i\bar{R}_i\|_{L_1}\sim \Ex|\sum_{i=0}^na_iR_i|$, where $R_i$ are products of
independent random variables distributed as $\bar{R}_1$.
Here is the more quantitative result.  

\begin{cor}
Suppose that $(n_k)_{k\geq 1}$ is an increasing sequence of positive integers such that $n_{k+1}/n_k\geq 3$ and
$\sum_{k=1}^{\infty}\frac{n_k}{n_{k+1}}<\infty$.
Then for any coefficients $a_0,a_1,\ldots,a_n$,
\begin{equation}
\label{L1Riesz}
c\sum_{i=0}^n|a_i|\leq \frac{1}{2\pi}\int_0^{2\pi}\Big|\sum_{i=0}^na_i\bar{R}_i(t)\Big|dt\leq \sum_{i=0}^n|a_i|,
\end{equation}
where $c>0$ is a positive constant that depends only on the sequence $(n_k)$. 
\end{cor}

\begin{proof}
We have $\bar{R}_i\geq 0$, so $\|\bar{R}_i\|_{L_1}=1$ and the upper estimate is obvious. To show the opposite bound
let $X_1,X_2,\ldots$ be independent random variables distributes as $1+\cos(Y)$, where $Y$ is uniformly
distributed on $[0,2\pi]$ and $R_i$ be as in \eqref{defR}. By the result of Y.~Meyer \cite{YM},
$\|\sum_{i=0}^na_i\bar{R}_i\|_{L_1}\geq c'\Ex|\sum_{i=0}^na_iR_i|$ and the lower estimate follows 
by Corollary \ref{iidprod1}.
\end{proof}

The condition $\sum_{k=1}^{\infty}\frac{n_k}{n_{k+1}}<\infty$ may be weakened to $\sum_{k=1}^{\infty}\frac{n_k^2}{n_{k+1}^2}<\infty$ \cite{MD},
we do not however know whether lower estimate holds under more general assumptions.
\medskip

\noindent
{\bf Problem.} Does the estimate \eqref{L1Riesz} holds for all sequences of integers such that $n_{k+1}/n_k\geq 3$?

\section{Proof of Theorem \ref{thm:noniid1}}

In this section $(F,\|\ \|)$ denotes a normed space. To avoid the measurability questions we assume that $F$ is finite
dimensional, in particular it is separable.
 First we show few simple estimates.

\begin{lem}
\label{singleest}
Suppose that $X$ is a nonnegative r.v. and $\Ex X=1$. Then for any $u,v\in F$ we have
\[
\Ex\|uX+v\|\geq \frac{1}{2}\Ex|X-1|\max\{\|u\|,\|v\|\}.
\]
\end{lem}

\begin{proof}
We have $\Ex\|uX+v\|\geq \|u\Ex X+v\|=\|u+v\|$. Moreover, 
\begin{align*}
\Ex\|uX+v\|&=\Ex\|u(X-1)+(u+v)\|\geq \|u\|\Ex|X-1|-\|u+v\|
\\
&\geq \|u\|\Ex|X-1|-\Ex\|uX+v\|
\end{align*}
and
\begin{align*}
\Ex\|uX+v\|&=\Ex\|v(1-X)+(u+v)X\|\geq \|v\|\Ex|X-1|-\|u+v\|\Ex|X|
\\
&\geq \|v\|\Ex|X-1|-\Ex\|uX+v\|.
\end{align*}
\end{proof}

\begin{lem}
\label{smallY}
Let $v\in F$ and $Y$ be a random vector with values in $F$ such that $\Pr(\|Y\|> \frac{\|v\|}{4})\leq 1/4$. Then
$\Ex\|Y+v\|\geq \Ex\|Y\|+\frac{\|v\|}{8}$.
\end{lem}

\begin{proof}
We have by the triangle inequality
\begin{align*}
\Ex\|Y+v\|&\geq \Ex(\|Y\|-\|v\|)\ind_{\{\|Y\|>\|v\|/4\}}+\Ex\Big(\|Y\|+\frac{\|v\|}{2}\Big)\ind_{\{\|Y\|\leq \|v\|/4\}}
\\
&=\Ex\|Y\|+\|v\|\Big(\frac{1}{2}\Pr\Big(\|Y\|\leq \frac{\|v\|}{4}\Big)-\Pr\Big(\|Y\|> \frac{\|v\|}{4}\Big)\Big)
\\
&\geq \Ex\|Y\|+\frac{\|v\|}{8}.
\end{align*}
\end{proof}

\begin{lem}
\label{sqrt}
Suppose that $X_i$ are independent nonnegative r.v's such that $\Ex\sqrt{X_i}\leq \lambda<1$ for all $i$.  
Then for any $v_0,\ldots,v_n\in F$,
\begin{equation}
\label{eq:sqrt1}
\Ex\bigg\|\sum_{k=0}^n v_kR_k\bigg\|^{1/2}\leq \sum_{k=0}^n \lambda^k\|v_k\|^{1/2}.
\end{equation}
and
\begin{equation}
\label{eq:sqrt2}
\Pr\bigg(\bigg\|\sum_{k=0}^n v_kR_k\bigg\|\geq \frac{t}{1-\lambda}\sum_{k=0}^n \lambda^k\|v_k\|\bigg)
\leq \frac{1}{\sqrt{t}}\quad
\mbox{ for }t\geq 1.
\end{equation}
\end{lem}

\begin{proof}
We have
\[
\Ex\bigg\|\sum_{k=0}^n v_kR_k\bigg\|^{1/2}\leq
\sum_{k=0}^n \Ex\|v_kR_k\|^{1/2}\leq \sum_{k=0}^n \lambda^k\|v_k\|^{1/2}.
\]
By the Cauchy-Schwarz inequality
\[
\bigg(\sum_{k=0}^n \lambda^k\|v_k\|^{1/2}\bigg)^2\leq 
\sum_{k=0}^n \lambda^k \sum_{k=0}^n \lambda^k\|v_k\|
\leq \frac{1}{1-\lambda}\sum_{k=0}^n \lambda^k\|v_k\|,
\]
and the estimate \eqref{eq:sqrt2} follows by \eqref{eq:sqrt1} and Chebyshev's
inequality.
\end{proof}

Now we are ready to formulate a main technical result that will easily imply Theorem \ref{thm:noniid1}.

\begin{prop}
\label{indest}
Let $X_1,X_2,\ldots$ satisfy assumption \eqref{ass1} and \eqref{ass2} and $0<\ve<\frac{1}{8}$ be such
that $\Pr(X_i\leq \ve)\geq p>0$ for all $i$.
Then for any vectors $v_0,v_1,\ldots,v_n\in F$  we have
\[
\Ex\bigg\|\sum_{i=0}^nv_iR_i\bigg\|\geq \alpha \|v_0\|+\sum_{k=1}^n(\beta-c_k)\|v_k\|,
\] 
where
\[
\alpha:=\frac{1}{16}p,\quad \beta:=\min\Big\{\frac{\alpha}{2},\frac{1}{32}\mu p\Big\} \quad \mbox{and}\quad
c_k:=\frac{4p\varepsilon}{1-\lambda}\sum_{i=0}^{k-1} \lambda^i.
\]
\end{prop}

\begin{proof}
We will proceed by induction on $n$. For $n=0$ the assertion is obvious, since $\alpha\leq 1$.

Now suppose that the induction assertion holds for $n$, we will show it for $n+1$. To this end
we consider two cases. To shorten the notation we put
\[
\tilde{R}_1:=1 \quad \mbox{and} \quad \tilde{R}_k:=\prod_{i=2}^kX_i \mbox{ for } k=2,3,\ldots. 
\]

\medskip
\noindent
{\bf Case 1.} $\|v_0\|\leq \frac{64\ve}{1-\lambda}\sum_{k=1}^{n+1}\lambda^{k-1}\|v_k\|$. 

By the induction assumption (applied conditionally on $X_1$)  we have
\begin{align*}
\Ex\bigg\|\sum_{i=0}^{n+1}v_iR_i\bigg\|&\geq \alpha \Ex\|v_0+v_1X_1\|+\sum_{k=2}^{n+1}(\beta-c_{k-1})\Ex\|X_1v_k\|
\\
&\geq \beta\|v_1\|+\sum_{k=2}^{n+1}(\beta-c_{k-1})\|v_k\|
\\
&\geq \alpha\|v_0\|-\frac{4p\ve}{1-\lambda}\sum_{k=1}^{n+1}\lambda^{k-1}\|v_k\|+\beta\|v_1\|
+\sum_{k=2}^{n+1}(\beta-c_{k-1})\|v_k\|
\\
&=\alpha\|v_0\|+\sum_{k=1}^{n+1}(\beta-c_k)\|v_k\|,
\end{align*}
where the second inequality follows by Lemma \ref{singleest}.

\medskip

\noindent
{\bf Case 2.} $\|v_0\|\geq \frac{64\ve}{1-\lambda}\sum_{k=1}^{n+1}\lambda^{k-1}\|v_k\|$.

The induction assumption, applied conditionally on $X_1$, yields
\begin{align}
\notag
\Ex\bigg\|\sum_{i=0}^{n+1}&v_iR_i\bigg\|\ind_{\{X_1>\ve\}}
\\
\label{largeX1}
&\geq \alpha \Ex\|v_0+v_1X_1\|\ind_{\{X_1>\ve\}}+\sum_{k=2}^{n+1}(\beta-c_{k-1})\Ex\|X_1v_k\|\ind_{\{X_1>\ve\}}.
\end{align}

Let $Y$ has the same distribution as $\sum_{i=1}^{n+1}v_iR_i=X_1\sum_{i=1}^{n+1}v_i\tilde{R}_i$ conditioned on the set 
$\{X_1\leq \ve\}$.
Then
\begin{align*}
\Pr\Big(\|Y\|>\frac{1}{4}\|v_0\|\Big)
&\leq \Pr\bigg(\ve\bigg\|\sum_{i=1}^{n+1}v_i\tilde{R}_i\bigg\|> \frac{1}{4}\|v_0\|\bigg)
\\
&\leq \Pr\bigg(\bigg\|\sum_{i=1}^{n+1}v_i\tilde{R}_i\bigg\|> \frac{16}{1-\lambda}\sum_{k=1}^{n+1}\lambda^{k-1}\|v_k\|\bigg)
\leq \frac{1}{4}   
\end{align*}
by Lemma \ref{sqrt}. Thus we may apply Lemma \ref{smallY} and get
\begin{align*}
\Ex\bigg\|\sum_{i=0}^{n+1}v_iR_i\bigg\|\ind_{\{X_1\leq \ve\}}&=\Pr(X_1\leq \ve)\Ex\|v_0+Y\|\geq
\Pr(X_1\leq \ve)\Big(\Ex\|Y\|+\frac{\|v_0\|}{8}\Big)
\\
&=\Ex\bigg\|\sum_{i=1}^{n+1}v_iR_i\bigg\|\ind_{\{X_1\leq \ve\}}+\frac{\|v_0\|}{8}\Pr(X_1\leq \ve).
\end{align*}

By the induction assumptions we get
\begin{align*}
\Ex\bigg\|\sum_{i=1}^{n+1}&v_iR_i\bigg\|\ind_{\{X_1\leq \ve\}}
\geq \alpha \Ex\|v_1X_1\|\ind_{\{X_1\leq \ve\}}+
\sum_{k=2}^{n+1}(\beta-c_{k-1})\Ex\|v_kX_1\|\ind_{\{X_1\leq \ve\}}
\\
&\geq \alpha \Ex\|v_0+v_1X_1\|\ind_{\{X_1\leq \ve\}}-\alpha\|v_0\|+
\sum_{k=2}^{n+1}(\beta-c_{k-1})\Ex\|v_kX_1\|\ind_{\{X_1\leq \ve\}}
\end{align*}
The above inequalities and our choice of $\alpha$ imply
\begin{align*}
\Ex\bigg\|\sum_{i=0}^{n+1}&v_iR_i\bigg\|\ind_{\{X_1\leq \ve\}}
\\
&\geq \alpha\Ex\|v_0+v_1X_1\|\ind_{\{X_1\leq \ve\}}+
\sum_{k=2}^{n+1}(\beta-c_{k-1})\Ex\|v_kX_1\|\ind_{\{X_1\leq \ve\}}+\alpha \|v_0\|.
\end{align*}

Together with \eqref{largeX1} this gives
\begin{align*}
\Ex\bigg\|\sum_{i=0}^{n+1}v_iR_i\bigg\|
&\geq \alpha\|v_0\|+\alpha\Ex\|v_0+v_1X_1\|+\sum_{k=2}^{n+1}(\beta-c_{k-1})\|v_k\|
\\
&\geq \alpha\|v_0\|+\beta\|v_1\|+\sum_{k=2}^{n+1}(\beta-c_{k-1})\|v_k\|
\\
&\geq \alpha\|v_0\|+\sum_{k=1}^{n+1}(\beta-c_{k})\|v_k\|,
\end{align*}
where the second inequality follows by Lemma \ref{singleest}.
\end{proof}

\begin{proof}[Proof of Theorem \ref{thm:noniid1}]
We apply Proposition \ref{indest} with $\ve:=\frac{(1-\lambda)^2}{256}\min\{\mu,1\}$ and 
$p:=\min_i\Pr(X_i\leq \ve)$. Notice that then 
$\beta = \frac{1}{32}\min\{\mu,1\}p\leq \alpha$ and we get
\begin{align*}
\Ex\bigg\|\sum_{i=0}^nv_iR_i\bigg\|&\geq \alpha\|v_0\|+
\sum_{i=0}^n (\beta-c_i)\|v_i\|\geq \Big(\beta-\frac{4p\ve}{(1-\lambda)^2}\Big)\sum_{i=0}^n \|v_i\|
\\
&\geq \frac{\beta}{2}\sum_{i=0}^n \|a_i\|.
\end{align*}
\end{proof}

\section{Proof of Theorem \ref{thm:noniid2}}

We start with a few refinements of lemmas from the previous section.

\begin{lem}
\label{lem:tech1}
Suppose that $X$ is nonnegative $\Ex X=1$, $\Ex|X-1|\geq\mu$ and $\Ex|X-1|\ind_{\{X>A\}}\leq \frac{1}{4}\mu$. 
Then
\[
\Ex\|uX+v\|\ind_{\{X\leq A\}}\geq \frac{1}{8}\mu\|v\|\quad \mbox{ for any }u,v\in F.
\]
\end{lem}

\begin{proof}
Let $Y$ has the same distribution as $X$ conditioned on the set $\{X\leq A\}$. Then $p:=\Ex Y\leq \Ex X= 1$ and
\[
\Ex\|uX+v\|\ind_{\{X\leq A\}}=\Pr(X\leq A)\Ex\|uY+v\|\geq \Pr(X\leq A)\|up+v\|.
\] 
We have $\Ex(X-1)_+=\Ex(X-1)_{-}\geq \frac{1}{2}\mu$, so
\[
\Pr(X\leq A)\Ex|Y-p|=\Ex|X-p|\ind_{\{X\leq A\}}\geq \Ex(X-1)_{+}\ind_{\{X\leq A\}}\geq \frac{1}{4}\mu
\] 
and
\begin{align*}
\Ex\|uY+v\|&=\frac{1}{p}\Ex\|v(p-Y)+(pu+v)Y\|\geq \|v\|\frac{1}{p}\Ex|Y-p|-\|pu+v\|\frac{1}{p}\Ex Y
\\
&\geq\frac{1}{4\Pr(X\leq A)}\mu\|v\|-\Ex\|uY+v\|.
\end{align*}
\end{proof}

\begin{lem}
\label{lem:tech2}
Let $Y$ and $Z$ be random vectors in $F$ such that 
\[
\Ex\|Z\|\ind_{\{\|Y\|>\frac{1}{8}\Ex \|Z\|\}}\leq \frac{1}{8}\Ex\|Z\|.
\]
Then $\Ex\|Y+Z\|\geq \Ex\|Y\|+\frac{1}{2}\Ex\|Z\|$.
\end{lem}

\begin{proof}
We have
\begin{align*}
\Ex\|Y+Z\|
&\geq \Ex(\|Y\|+\|Z\|-2\|Z\|)\ind_{\{\|Y\|>\frac{1}{8}\Ex\|Z\|\}}
\\
&\phantom{>}+\Ex(\|Y\|+\|Z\|-2\|Y\|)\ind_{\{\|Y\|\leq\frac{1}{8}\Ex\|Z\|\}}
\\
&=\Ex\|Y\|+\Ex\|Z\|-2\Ex\|Z\|\ind_{\{\|Y\|>\frac{1}{8}\Ex\|Z\|\}}-2\Ex\|Y\|\ind_{\{\|Y\|\leq\frac{1}{8}\Ex \|Z\|\}}
\\
&\geq \Ex\|Y\|+\Ex\|Z\|-\frac{2}{8}\Ex\|Z\|-\frac{2}{8}\Ex\|Z\|=\Ex\|Y\|+\frac{1}{2}\Ex\|Z\|.
\end{align*}
\end{proof}

\begin{lem}
\label{lem:tech3}
Suppose that $X_1,\ldots,X_n$ are independent, nonnegative and $\Ex|X_i-1|\geq \mu$. Then for any vectors
$v_0,\ldots,v_n\in F$,
\[
\Ex\Big\|\sum_{i=0}^n v_iR_i\Big\|\geq \frac{1}{4}\mu^2\max\{\|v_0\|,\ldots,\|v_n\|\}.
\]
In particular
\[
\Ex\Big\|\sum_{i=0}^k v_iR_i\Big\|\geq \frac{1}{4k}\mu^2\sum_{i=1}^k\|v_i\|. 
\]
\end{lem}

\begin{proof}
We have for any $0\leq j\leq n$,
$\sum_{i=0}^n v_iR_i=Y+X_j(v_jR_{j-1}+X_{j+1}Z)$, where variables $Y$ and $Z$ are independent
of $X_j$ and $X_{j+1}$. So Lemma \ref{singleest} applied conditionally yields
\begin{align*}
\Ex\Big\|\sum_{i=0}^n v_iR_i\Big\|
&\geq \frac{1}{2}\Ex|X_j-1|\Ex\|v_jR_{j-1}+X_{j+1}Z\|
\\
&\geq\frac{1}{2}\Ex|X_j-1|\frac{1}{2}\Ex|X_{j+1}-1|\Ex\|v_jR_{j-1}\|
\geq \frac{1}{4}\mu^2\|v_j\|.
\end{align*}
\end{proof}

Next statement is a variant of Proposition \ref{indest}.

\begin{prop}
\label{prop:indest2}
Let $X_1,X_2,\ldots$ satisfy assumption \eqref{ass1}-\eqref{ass3} and $k\geq 1$.
Then for any vectors $v_0,v_1,\ldots,v_n\in F$ and $\ve>0$ we have
\[
\Ex\bigg\|\sum_{i=0}^nv_iR_i\bigg\|\geq \alpha \|v_0\|+\sum_{i=1}^n(\beta-c_i)\|v_i\|,
\] 
where
\[
\alpha:=\frac{1}{64}\mu,\quad \beta:=\frac{1}{4k}\mu^2\alpha, \quad c_i:=0 \mbox{ for } 1\leq i\leq k-1
\]
and
\[
c_i:=\frac{2^{8}A}{1-\lambda}\sum_{j=k}^i\lambda^{j+k-2},\quad \mbox{ for }  i=k,k+1,\ldots.
\]
\end{prop}

\begin{proof}
Observe that $\mu\leq 2$, hence $\alpha\leq \frac{1}{32}$ and 
$\beta\leq\min\{\frac{1}{8k}\mu^2,\frac{\alpha}{2}\mu\}$. 
As before we will proceed by induction on $n$. Notice that by Lemmas \ref{singleest} and \ref{lem:tech3} we have for $n\leq k$,
\[
\Ex\bigg\|\sum_{i=0}^nv_iR_i\bigg\|\geq \frac{1}{4}\mu\|v_0\|+\frac{1}{8k}\mu^2\sum_{i=1}^n\|v_i\|
\geq \alpha\|v_0\|+\sum_{i=1}^n\beta\|v_i\|.
\] 

Now suppose that the induction assertion holds for $n\geq k$, we will show it for $n+1$. To this end
we consider two cases. To shorten the notation we put
\[
R_{k+1,k}:=1 \quad \mbox{and} \quad R_{k+1,l}:=\prod_{i=k+1}^lX_i \mbox{ for } l\geq k+1. 
\]

\medskip
\noindent
{\bf Case 1.} $\mu\|v_0\|\leq \frac{2^{14}}{1-\lambda}A\sum_{i=k}^{n+1}\lambda^{i+k-2}\|v_i\|$. 

By the induction assumption (applied conditionally on $X_1$)  we have
\begin{align*}
\Ex\bigg\|\sum_{i=0}^{n+1}v_iR_i\bigg\|&\geq \alpha \Ex\|v_0+v_1X_1\|+\sum_{i=2}^{n+1}(\beta-c_{i-1})\Ex\|X_1v_i\|
\\
&\geq \beta\|v_1\|+\sum_{i=2}^{n+1}(\beta-c_{i-1})\|v_i\|
\\
&\geq \alpha\|v_0\|-\frac{2^{8} A}{1-\lambda}\sum_{i=k}^{n+1}\lambda^{i+k-2}\|v_i\|
+\beta\|v_1\|+\sum_{i=2}^{n+1}(\beta-c_{i-1})\|v_i\|
\\
&=\alpha\|v_0\|+\sum_{i=1}^{n+1}(\beta-c_i)\|v_i\|,
\end{align*}
where the second inequality follows by Lemma \ref{singleest}.

\medskip
\noindent
{\bf Case 2.} $\mu\|v_0\|\geq \frac{2^{14}}{1-\lambda}A\sum_{i=k}^{n+1}\lambda^{i+k-2}\|v_i\|$. 

Define the event $A_k\in \sigma(X_1,\ldots,X_k)$ by
\[
A_k:=\{X_1\leq A,\ R_{2,k}\leq 4\lambda^{2k-2}\}.
\]
By the induction assumption (applied conditionally) we have
\begin{equation}
\label{eq:estind1}
\Ex\bigg\|\sum_{i=0}^{n+1}v_iR_i\bigg\|\ind_{\Omega\setminus A_k}\geq
\alpha \Ex\bigg\|\sum_{i=0}^kv_iR_i\bigg\|\ind_{\Omega\setminus A_k}
+\sum_{i=k+1}^{n+1}(\beta-c_{i-k})\Ex\|v_iR_k\|\ind_{\Omega\setminus A_k}.
\end{equation}

We have
\[
\Ex\bigg\|\sum_{i=0}^{n+1}v_iR_i\bigg\|\ind_{A_k}=\Pr(A_k)\Ex\|Y+Z\|,
\]
where $Y$ has the same distribution as the random variable $\sum_{i=k}^{n+1}v_iR_i$ conditioned on the event $A_k$
and $Z$ has the same distribution as the random variable $\sum_{i=0}^{k-1}v_iR_i$ conditioned on the event $A_k$.
Lemma \ref{lem:tech1} applied conditionally implies
\[
\Ex\|Z\|\geq \frac{1}{\Pr(X_1\leq A)}\frac{1}{8}\mu\|v_0\|\geq \frac{1}{8}\mu\|v_0\|.
\] 
Notice also that 
\[
\|Y\|=\|R_kY'\|\leq 4A\lambda^{2k-2}\|Y'\|,
\]
where $Y'$ is independent of $Z$ with the same distribution as $\sum_{i=k}^{n+1}v_{i}R_{k+1,i}$. Therefore
\begin{align*}
\Ex\|Z\|\ind_{\{\|Y\|\geq \frac{1}{8}\Ex\|Z\|\}}
&\leq \Ex\|Z\|\ind_{\{64\|Y\|\geq \mu\|v_0\|\}}
\leq \Ex\|Z\|\ind_{\{256A\lambda^{2k-2}\|Y'\|\geq \mu\|v_0\|\}}
\\
&=\Ex\|Z\|\Pr(256A\lambda^{2k-2}\|Y'\|\geq \mu\|v_0\|).
\end{align*} 
We have (by our assumptions on $v_0$)
\begin{align*}
\Pr(256A\lambda^{2k-2}\|Y'\|\geq \mu&\|v_0\|)
\leq \Pr\Big(\|Y'\|\geq \frac{2^6}{1-\lambda}\sum_{i=k}^{n+1}\lambda^{i-k}\|v_i\|\Big)
\\
&=\Pr\Big(\Big\|\sum_{i=k}^{n+1}v_{i}R_{k+1,i}\Big\|\geq \frac{2^6}{1-\lambda}\sum_{i=k}^{n+1}\lambda^{i-k}\|v_i\|\Big)
\leq \frac{1}{8},
\end{align*}
where the last inequality follows by Lemma \ref{sqrt}. Thus 
$\Ex\|Z\|\ind_{\{\|Y\|\geq \frac{1}{8}\Ex\|Z\|\}}\leq \frac{1}{8}\Ex\|Z\|$
and by Lemma \ref{lem:tech2}, $\Ex\|Z+Y\|\geq \Ex\|Y\|+\frac{1}{2}\Ex\|Z\|$, that is
\begin{equation}
\label{eq:estn1}
\Ex\bigg\|\sum_{i=0}^{n+1}v_iR_i\bigg\|\ind_{A_k}\geq 
\frac{1}{2}\Ex\bigg\|\sum_{i=0}^{k-1}v_iR_i\bigg\|\ind_{A_k}+\Ex\bigg\|\sum_{i=k}^{n+1}v_iR_i\bigg\|\ind_{A_k}.
\end{equation}
By Lemma \ref{lem:tech1}
\[
\Ex\bigg\|\sum_{i=0}^{k-1}v_iR_i\bigg\|\ind_{A_k}\geq\frac{1}{8}\mu\|v_0\|\Pr(R_{2,k}\leq 4\lambda^{2k-2})\geq
\frac{1}{16}\mu\|v_0\|=4\alpha\|v_0\|,
\]
where the second inequality follows by the bound
$\Ex\sqrt{R_{2,k}}=\prod_{i=2}^{k}\Ex\sqrt{X_i}\leq \lambda^{k-1}$ and Chebyshev's inequality. 
Since $\alpha\leq \frac{1}{4}$ we get
\begin{equation}
\label{eq:estn2}
\frac{1}{2}\Ex\bigg\|\sum_{i=0}^{k-1}v_iR_i\bigg\|\ind_{A_k}
\geq \alpha\|v_0\|+\alpha\Ex\bigg\|\sum_{i=0}^{k-1}v_iR_i\bigg\|\ind_{A_k}.
\end{equation}
By the induction assumption
\begin{equation}
\label{eq:estn3}
\Ex\bigg\|\sum_{i=k}^{n+1}v_iR_i\bigg\|\ind_{A_k}
\geq \alpha\Ex\|v_kR_k\|\ind_{A_k}+\sum_{i=k+1}^{n+1}(\beta-c_{i-k})\Ex\|v_iR_k\|\ind_{A_k}.
\end{equation}
By \eqref{eq:estn1}-\eqref{eq:estn3} we get
\[
\Ex\bigg\|\sum_{i=0}^{n+1}v_iR_i\bigg\|\ind_{A_k}
\geq \alpha\|v_0\|+\alpha\Ex\bigg\|\sum_{i=0}^{k}v_iR_i\bigg\|\ind_{A_k}+
\sum_{i=k+1}^{n+1}(\beta-c_{i-k})\Ex\|v_iR_k\|\ind_{A_k}.
\]
Together with \eqref{eq:estind1} this yields
\begin{align*}
\Ex\bigg\|&\sum_{i=0}^{n+1}v_iR_i\bigg\|
\geq \alpha\|v_0\|+\alpha\Ex\bigg\|\sum_{i=0}^{k}v_iR_i\bigg\|
+\sum_{i=k+1}^{n+1}(\beta-c_{i-k})\Ex\|v_iR_k\|
\\
&\geq \alpha\|a_0\|+\beta\sum_{i=1}^k\|v_i\|+\sum_{i=k+1}^{n+1}(\beta-c_{i-k})\|v_i\|
\geq \alpha\|v_0\|+\sum_{i=1}^{n+1}(\beta-c_{i})|v_i\|,
\end{align*}
where the second inequality follows by Lemma \ref{lem:tech3} and the definition of $\beta$.
\end{proof}

\begin{proof}[Proof of Theorem \ref{thm:noniid2}.]
Let $\alpha,\beta$ and $c_i$ be as in Proposition \ref{prop:indest2}. Observe that \eqref{def:k} yields
\[
c_i\leq \frac{2^8}{(1-\lambda)^2}\lambda^{2k-2}A
\leq \frac{\beta}{2}=2^{-9}\frac{\mu^3}{k}
\]
therefore $\alpha,\beta-c_i\geq \frac{1}{2}\beta=\frac{1}{512k}\mu^3$ for all $i$ and the assertion follows
by Proposition \ref{prop:indest2}.
\end{proof}

\noindent
Institute of Mathematics\\
University of Warsaw\\
Banacha 2\\
02-097 Warszawa\\
Poland\\
\texttt{rlatala@mimuw.edu.pl}

\end{document}